\documentclass[a4paper,onecolumn,11pt]{quantumarticle}
\pdfoutput=1
\usepackage[utf8]{inputenc}
\usepackage{amsmath,amsthm}
\usepackage{amssymb}
\usepackage{bbold}
\usepackage{xparse}
\usepackage{physics}
\usepackage{hyperref}
\usepackage{enumitem}

\newtheorem*{theorem*}{Theorem}

\newcommand{\ca}{\mathcal A}

\newcommand{\ch}{\mathcal H}

\newcommand{\cl}{\mathcal L}

\newcommand{\tp}{\tilde{P}}

\def\be{\begin{equation}}
\def\ee{\end{equation}}
\def\ba{\begin{align}}
\def\ea{\end{align}}

\newcommand{\id}{\mathbb{1}}

\begin{document}

\title{A self-contained proof of the Artin-Wedderburn theorem in the case of finite-dimensional Von Neumann algebras}

\author{Octave Mestoudjian}

\affiliation{Université Paris-Saclay, Inria, CNRS, LMF, 91190 Gif-sur-Yvette, France}

\author{Pablo Arrighi}

\affiliation{Université Paris-Saclay, Inria, CNRS, LMF, 91190 Gif-sur-Yvette, France}

\begin{abstract}
We provide a self-contained proof of the Artin-Wedderburn theorem in the case of finite-dimensional Von Neumann algebras (or equivalently unital C* algebras) that is fully constructive and uses only basic notions of linear algebra.
\end{abstract}

\maketitle

The Artin-Wedderburn theorem, originally stated in \cite{artin} and \cite{wedderburn}, is a fundamental theorem on the strucutre of (semi-simple) rings. As a special case, it provides a result on the decomposition of finite-dimensional Von Neumann algebras (or equivalently unital C* algebras), as stated for example in Theorem 11.2 of \cite{takesaki}. Some self-contained proofs have been given for the general result in \cite{brevsar2010}, \cite{brevsar2024} and \cite{gao}. Following similar ideas, we provide a proof in the special case of Von Neumann algebras that has the advantage of being fully constructive and to use only basic notions of linear algebra.

\begin{theorem*}[Artin-Wedderburn]
Let $\ch$ be a finite-dimensional Hilbert space and $\ca \subseteq \cl(\ch)$ be a Von Neumann algebra. Then there exists families of Hilbert spaces $(\ch_L^{k})$ and $(\ch_R^{k})$ and a unitary $U: \ch \rightarrow \bigoplus_k (\ch_L^{k} \otimes \ch_R^{k})$ such that 
\begin{equation}
U \ca U^{\dagger} = \bigoplus_k \cl(\ch_L^{k}) \otimes \id_{\ch_R^{k}}.
\end{equation}
\end{theorem*}

\begin{proof}
Let $P=\{P_i\}_{i \in \{1, \ldots, p \}}$ be a family of non-zero, orthogonal, and pairwise orthogonal projectors of $\mathcal{A}$, ie. such that: 

\begin{enumerate}[label = (\roman*)]
\item $\forall i \in \{1, \ldots, p \}, \; 0 \neq P_i = P_i^{\dagger} \in\mathcal{A}$; 
\item $\forall i,j\in \{1, \ldots, p \} \; P_iP_j=\delta_{ij} P_i \in \mathcal{A}$. 
\end{enumerate}
Notice that such a family always exists as $\{ \id_{\ch} \} \subseteq \mathcal{A}$ is one. Moreover we will take $P$ to be maximal among such families, i.e. so that there is no family $Q=\{Q_j\}$ verifying conditions $(i),(ii)$ and such that $\mathcal{P}\subsetneq \mathcal{Q}$, with $\mathcal{P},\mathcal{Q}$ the algebras spanned by $P$ and $Q$. It follows that
\begin{enumerate}[label = (\roman*)]
\item[(iii)] $\sum_{i \in \{1, \ldots, p \}} P_i = \id$,
\end{enumerate}
otherwise $\id - \sum_{i \in \{1, \ldots, p \}} P_i$ could be added to the set of projectors which would contradict the maximality of $P$.\vspace{1mm}\\

\newpage

\textsc{First}, we show that
\begin{align}
\forall j, \quad P_j\mathcal{A}P_j=\mathbb{C}P_j.\label{diagcase}
\end{align}
Intuitively, this is because the contrary would allow us to refine the subspaces defined by $P_j$ into
smaller subspaces, and hence go against the fact that $P$ is maximal. Remark that we always have $\mathbb{C}P_j \subseteq P_j\mathcal{A}P_j$, because $P_j P_j P_j = P_j \neq 0$ and that the reverse inclusion is immediately true if $P_j$ is of rank one. Let us now assume that $P_j$ is of rank strictly greater than one, and suppose that there exists some (necessarily non-zero) $M\in\mathcal{A}$ such that $P_j M P_j \not\propto P_j$. Both $P_j(M + M^{\dagger})P_j$ and $iP_j(M-M^{\dagger})P_j$ are hermitian, and at least one of them is not a multiple of $P_j$ (otherwise $P_j M P_j$ would also be one); we will call it $H$. Remark that $H$ has support in the subspace $Im(P_j)$, and hence that it can be decomposed as $H=\sum_k \lambda_k Q_k$ where the $Q_k$'s are orthogonal and pairwise orthogonal projectors, such that $\sum_k Q_k = P_j$, and the $\lambda_k$'s are distinct real numbers. This sum has at least two terms as $P_j$ is of rank strictly greater than one and $H \not\propto P_j$. All the $Q_k$'s are in $\mathcal{A}$, as for any $k$ such that $\lambda_k \neq 0$, $Q_k=\frac{1}{\lambda_k} H\prod_{l\neq k} (H-\lambda_l P_j) \in \ca$ (and the one such that $\lambda_k = 0$ can be obtained by substracting from $P_j$ all the other $Q_k$). Consider now the set $Q=P/\{P_j\} \cup (\bigcup_k \{Q_k\})$. It satisfies conditions $(i)$ and $(ii)$, is different from $P$ and such that $\mathcal{P}\subsetneq \mathcal{Q}$, which contradicts the maximality of $P$ and concludes the proof of this first property. \vspace{1mm}\\

\textsc{Second}, we can define the following equivalence relation on (the indices of) the projectors: 
\begin{align}
i \sim j \Leftrightarrow P_i\mathcal{A}P_j \neq 0.
\end{align}
 It is indeed an equivalence relation, as it is: 
\begin{itemize}
\item reflexive, because $P_i P_i P_i = P_i \neq 0$;
\item transitive, indeed let $A,B \in \mathcal{A}$ such that $M = P_i A P_j \neq 0$ and $N = P_j B P_k \neq 0$. By Eq. (\ref{diagcase}), we know that there exists $\lambda,\mu \in \mathbb{C}$ such that $M^{\dagger}M = P_j A^{\dagger} P_i A P_j = \lambda P_j$ and $NN^{\dagger} = P_j B P_k B^{\dagger} P_j = \mu P_j$. Moreover, as $M \neq 0 \Leftrightarrow MM^{\dagger}  \neq 0$ and $N \neq 0 \Leftrightarrow NN^{\dagger} \neq 0$ it follows that $\lambda,\mu \neq 0$. Finally $P_j A^{\dagger} P_i A P_j B P_k B^{\dagger} P_j = \lambda \mu P_j \neq 0 \Rightarrow P_i A P_j B P_k \neq 0$;
\item symmetric, since
\begin{align*} 
P_i\mathcal{A}P_j\neq 0 & \Leftrightarrow (P_i\mathcal{A}P_j)^{\dagger}\neq 0 \\
& \Leftrightarrow P_j^{\dagger}\mathcal{A}^{\dagger}P_i^{\dagger} \neq 0\\
& \Leftrightarrow P_j\mathcal{A}P_i \neq 0.
\end{align*}
\end{itemize}
Let $I_1,...,I_n$ be the different equivalence classes and let us denote $P_{I_k} = \sum_{i \in I_k} P_i$. It follows, by definition of $\sim$, that 
\begin{align*}
\mathcal{A} = (\sum_{k = 1}^n P_{I_k}) \mathcal{A}  (\sum_{l = 1}^n P_{I_l}) = \sum_{k = 1}^n P_{I_k} \mathcal{A} P_{I_k}.
\end{align*}
This sum is actually direct as, for any $P_{I_k}A_kP_{I_k} = P_{I_l}A_lP_{I_l} = A \in P_{I_k} \mathcal{A}  P_{I_k} \cap P_{I_l} \mathcal{A}  P_{I_l}$, we have that $A = P_{I_k}A_kP_{I_k} = P_{I_k}P_{I_k}A_kP_{I_k} = P_{I_k}P_{I_l}A_lP_{I_l} = \delta_{kl} A$. Moreover, each $ P_{I_k} \mathcal{A}  P_{I_k}$ is (unitarily isomorphic to) a Von Neumann subalgebra of $\mathcal{L}(\mathcal{H}_{I_k})$ and the $\{P_i\}_{i \in {I_k}}$ form a maximal family of orthogonal and pairwise orthogonal projectors of this algebra. We thus obtain that
\begin{align*}
\mathcal{H} &= \bigoplus_k \mathcal{H}_{I_k}, \\
\ca & = \bigoplus_k \ca_{I_k},
\end{align*}
where $\ch_{I_k} = Im(P_{I_k})$ and $\ca_{I_k} = P_i \ca P_i$. \vspace{1mm}\\

In the rest of the proof we place ourselves in one of these algebras $\ca_{I_k}$, or equivalently one of the classes $I_k$, where any two projectors are such that 
\begin{align}\label{neq0}
P_i \mathcal{A} P_j \neq 0,
\end{align}
and will prove that $\mathcal{A}_{I_k}$ is unitarily isomorphic to some $\cl(\ch_L^{k}) \otimes \id_{\ch_R^{k}}$. As $\ca = \bigoplus_k \ca_{I_k}$, this will prove the result of the theorem. \vspace{1mm}\\

\textsc{Third}, we show that for all $A \in \ca$ (which we remind is now one of the $\ca_{I_k}$) and for all $i,j \in \{1\ldots p\}$, 
if we let $M=P_i A P_j$, then
\begin{align}
\exists \lambda\in\mathbb{C},\quad M^\dagger M=\lambda P_i\,\wedge\, M M^\dagger=\lambda P_j . \label{sympropto}
\end{align}
Indeed, by Eq. (\ref{diagcase}), we have that there exists $\lambda,\mu \in \mathbb{C}$ such that $M^\dagger M=\lambda P_i$ and $M M^\dagger=\mu P_j$. But then $\lambda^2 P_i = M^\dagger M M^\dagger M = \mu M^\dagger P_j M=  \mu\lambda P_i$, 
hence $\lambda$ equals $\mu$.\vspace{1mm}\\

\textsc{Fourth}, we show that 
\begin{align}
\forall i,j, \quad \trace(P_i)=\trace(P_j)=\textrm{some constant }q.\label{qvalue}
\end{align}
For each $i,j$ take some $A\in\mathcal{A}$ verifying Eq. (\ref{neq0}) and let $M=P_i A P_j$. By Eq. (\ref{sympropto}) there exists a complex number $\lambda$
such that $P_i=\lambda MM^{\dagger}$ and $P_j=\lambda M^{\dagger}M$.
Then the equality follows from  
$\trace(\lambda MM^{\dagger})=\trace(\lambda M^{\dagger}M) = q$ and we conclude that all the $P_i$ have the same rank.\vspace{1mm}\\

\textsc{Fifth}, consider some Hilbert spaces, $\ch_L $ of dimension the number $p$ of projectors in the family $P$, and $\ch_R$ of dimension $q$ the rank of these projectors. Define the unitary $V : \ch \rightarrow \ch_L \otimes \ch_R$ which takes each $P_i$ to $ \tilde{P}_i =\ketbra{i}{i} \otimes \id_{\ch_R} \in \cl(\ch_L \otimes \ch_R)$.
Note that this is always possible since the $\{P_i\}_{i \in \{1,\ldots, p\} }$ form an 
orthogonal $(ii)$, complete $(iii)$ set of projectors of equal rank by Eq. (\ref{qvalue}). 
We now show that 
\begin{align}
\forall B \in V\mathcal{A}V^{\dagger}, \, &\forall i,j, \quad  \tilde{P}_i B \tilde{P}_j =\ketbra{i}{j}\otimes B_{ij} \nonumber\\
&\textrm{ where }B_{ij}B_{ij}^{\dagger}= B_{ij}^{\dagger}B_{ij}\propto \id_{\ch_R}.\label{localunitary}
\end{align}
Indeed, the first line is true because
\begin{align*}
\tilde{P}_i B \tilde{P}_j &=\sum_{k,l,m,n} \alpha_{klmn} (\ketbra{i}{i}\otimes \id_{\ch_R})(\ketbra{k}{l}\otimes\ketbra{m}{n}) (\ketbra{j}{j}\otimes \id_{\ch_R})\\
&= \sum_{m,n} \alpha_{ijmn} (\ketbra{i}{j}\otimes\ketbra{m}{n}) = \ketbra{i}{j}\otimes (\sum_{mn} \alpha_{ijmn} \ketbra{m}{n}). 
\end{align*}
Now if we let $M= \tilde{P}_i B \tilde{P}_j$, by Eq. (\ref{sympropto}) there exists a complex number $\lambda$
such that we have both $\tilde{P}_i = \lambda MM^{\dagger}$ and $\tilde{P}_j = \lambda M^{\dagger}M$. It follows that 
\begin{align*}
\ketbra{i}{i}\otimes \id_{\ch_R} = \tilde{P}_i  = \lambda MM^{\dagger}&=\lambda(\ketbra{i}{j}\otimes B_{ij})(\ketbra{j}{i}\otimes B_{ij}^{\dagger})\\
&=\lambda\ketbra{i}{i}\otimes B_{ij}B_{ij}^{\dagger}
\end{align*}
and hence that $\lambda B_{ij}B_{ij}^{\dagger} = \id_{\ch_R}$. The result for $B_{ij}^{\dagger}B_{ij}$ is proved symetrically. \vspace{1mm}\\

\textsc{Sixth}, by Eq. (\ref{neq0}), one can find for each $i$ an element $A^{(i)} \in \ca$ such that $P_1 A^{(i)} P_i \neq 0$ and which can be rescaled so that Eq. (\ref{localunitary}) on $N^{(i)} = V A^{(i)} V^{\dagger}$ makes $N^{(i)}_{1i}$ unitary. We call $A = \sum_i P_1 A^{(i)}P_i \in \ca$ and $B = V A V^{\dagger}$ which is then such that for all $i$, $B_{1i} = N^{(i)}_{1i}.$ We now define the unitary 
\begin{align*}
W&=\sum_i(\ketbra{i}{i}\otimes B_{1i})\\
&=\sum_i (\ketbra{i}{1} \otimes \id_{\ch_R})(\ketbra{1}{i}\otimes B_{1i})\\
&=\sum_i (\ketbra{i}{1} \otimes \id_{\ch_R}) \tilde{P_1} B \tilde{P}_i ,
\end{align*}
which has the property that for all $j$,
\begin{align*}
W \tilde{P_j} W^{\dagger} &= \sum_{i,l} (\ketbra{i}{1} \otimes \id_{\ch_R}) \tilde{P_1} B \tilde{P}_i \tp_j \tp_l B^{\dagger} \tp_1 (\ketbra{1}{l} \otimes \id_{\ch_R}) \\
& = (\ketbra{j}{1} \otimes \id_{\ch_R}) \tilde{P_1} B \tp_j  B^{\dagger} \tp_1 \ketbra{1}{j} \otimes \id_{\ch_R}) \\
& = (\ketbra{j}{1} \otimes \id_{\ch_R}) B_{1j} B_{1j}^{\dagger} \ketbra{1}{j} \otimes \id_{\ch_R}) \\
& = (\ketbra{j}{j} \otimes \id_{\ch_R}) \\
& = \tp_j.
\end{align*}

\textsc{Finally}, we define the unitary $U = WV : \ch \rightarrow \ch_L \otimes \ch_R$ and we show that, for all $A \in \ca$ and for all $i,j$, there exists $\lambda \in \mathbb{C}$ such that 
\begin{align}
 \tilde{P}_i (U A U^{\dagger}) \tilde{P}_j = \lambda \ketbra{i}{j}\otimes \id_{\ch_R};
\end{align}
which will conclude the proof, as $ U A U^{\dagger} = \sum_{i,j}  \tilde{P}_i ( U A U^{\dagger}) \tilde{P}_j$, and for any $i,j$ one can find $A \in \ca$ such that  $\tilde{P}_i  (U A U^{\dagger}) \tilde{P}_j  = U P_i A P_j U^{\dagger}$ is non-zero (by Eq. (\ref{neq0})). And indeed this equality is true as
\begin{align*}
& \tilde{P}_i  U A U^{\dagger} \tilde{P}_j \\
=& \tilde{P}_i W V  A V^{\dagger} W^{\dagger} \tilde{P}_j \\
=&  W \tp_i V A V^{\dagger} \tp_j W^{\dagger} \\
=&  \left( \sum_{k} (\ketbra{k}{1} \otimes \id_{\ch_R}) \tilde{P}_1 B \tilde{P}_k \right) \tp_i U A U^{\dagger} \tp_j \left( \sum_l \tp_l B^{\dagger} \tp_1 (\ketbra{1}{l} \otimes \id_{\ch_R}) \right) \\
=& (\ketbra{i}{1} \otimes \id_{\ch_R}) \tilde{P_1} B  \tp_i U A U^{\dagger} \tp_j B^{\dagger} \tp_1 (\ketbra{1}{j} \otimes \id_{\ch_R})\\
=& (\ketbra{i}{1} \otimes \id_{\ch_R}) \lambda \tp_1 (\ketbra{1}{j} \otimes \id_{\ch_R}) \\
=&  \lambda (\ketbra{i}{1}\otimes\id_{\ch_R}) (\ketbra{1}{1}\otimes \id_{\ch_R}) (\ketbra{1}{j}\otimes \id_{\ch_R}) \\
=& \lambda \ketbra{i}{j}\otimes\id_{\ch_R}.
\end{align*}

\end{proof}

\bibliographystyle{plain}
\bibliography{refs}

\end{document}